\documentclass[preprint,12pt,3p]{elsarticle}

\let\today\relax
\makeatletter
\def\ps@pprintTitle{%
    \let\@oddhead\@empty
    \let\@evenhead\@empty
    \def\@oddfoot{\footnotesize\itshape
         {Submitted preprint} \hfill\today}%
    \let\@evenfoot\@oddfoot
    }
\makeatother

\usepackage{lineno,hyperref}
\modulolinenumbers[5]

\usepackage{amsmath,amssymb}
\usepackage{amsthm}
\usepackage{graphicx}

\newtheorem{theorem}{Theorem}

\newtheorem{proposition}{Proposition}

\newtheorem{definition}{Definition}
\newtheorem{remark}{Remark}[section]

\newcommand{\R}{\mathbb{R}}
\newcommand{\St}{\mathbb{S}}

\def\spark{\textnormal{spark\,}}
\def\rank{\textnormal{rank\,}}

\def\tr{\textnormal{tr\,}}

\DeclareMathOperator*{\argmin}{arg\,min}

\usepackage{color}

\begin{document}

\begin{frontmatter}

\title{Reconstruction of jointly sparse vectors via manifold optimization}

\author[address1]{Armenak Petrosyan}
\ead{petrosyana@ornl.gov}

\author[address1]{Hoang Tran}
\ead{tranha@ornl.gov}

\author[address1,address2]{Clayton Webster}
\ead{webstercg@ornl.gov}

\address[address1]{Department of Computational and Applied Mathematics, Oak Ridge National Laboratory, 1 Bethel Valley Road, P.O. Box 2008, Oak Ridge TN 37831-6164, USA}
\address[address2]{Department of Mathematics, University of Tennessee, 1403 Circle Drive, Knoxville, TN 37916, USA}

\begin{abstract}
In this paper, we consider the challenge of reconstructing jointly sparse vectors from linear measurements. 
Firstly, we show that by utilizing the rank of the output data matrix we can reduce the problem to a full column rank case.  
This result reveals a reduction in the computational complexity of the original problem and enables a simple implementation 
of joint sparse recovery algorithms for full-rank setting.  Secondly, we propose a new method for joint sparse recovery in the form of a non-convex optimization problem on a non-compact Stiefel manifold. 
In our numerical experiments our method outperforms the commonly used $\ell_{2,1}$ minimization in the sense that much fewer measurements are required for accurate sparse reconstructions. We postulate this approach possesses the desirable rank aware property, that is, being able to take advantage of the rank of the unknown matrix to improve the recovery. 
\end{abstract}


\end{frontmatter}

 \section{Introduction}
 \label{sec:intro}
The sparse sampling and recovery problem  has been a subject of intensive research for the last two decades, especially after major advances in the field of compressed sensing, through the seminal works \cite{candes2006robust,Donoho06}.  In the classical problem, the goal is to find the vector 
$x\in \R^N$ from the system
$Ax=y$ where $A\in\R^{M\times N}$, under the assumption that the vector $x$ is sparse, i.e., has only few non-zero entries.
{Here, $N$ represents the dimensionality of the unknown vector $x$ and $M$ is the number of linear measurements performed on $x$. A typical difficulty in solving the problem stems from the fact that one often has only a limited number of measurements, i.e., $M<N$, so the system is underdetermined.}

In the joint sparse recovery problem, { also known as multiple measurements vectors (MMV), simultaneous sparse approximation, or collaborative sparse coding,} we instead solve a matrix equation $AX=Y$ where the unknown matrix $X\in \R^{N\times K}$ has only few non-zero rows. {$X$ can be viewed as a collection of $K$ unknown vectors in $\mathbb{R}^N$ which share the same sparsity support. In this case, it has been demonstrated \cite{Schmidt79,Feng98,mishali2008reduce} that recovery of multiple vectors simultaneously can take advantage of their joint sparsity, which helps reduce the number of measurements compared to individual reconstructing each at a time.} The joint sparse recovery problem has applications in hyper-spectral imaging \cite{ chen2011hyperspectral,fang2014spectral,wang2017joint,ZHU2014101}, source localization \cite{van2011sparse,malioutov2005sparse} and many other fields.
{ 
In these problems, the joint sparse recovery is often conducted under additional positivity constraints, which are out of the scope of this paper. Our proposed method therefore is not applicable directly, but is an important block towards these applications. Another example is the CLSUnSAL algorithm \cite{iordache2014collaborative} for hyperspectral unmixing, which is based on a joint sparse regularization.
}

In this paper, we first utilize the rank of the output data matrix $Y$ to reduce the problem to the case where unknown matrix has full column rank. More specifically, we prove the following theorem (Section \ref{sec:reduce}).
\begin{theorem}
Denote by $\Sigma_{s,r}$ the set of all matrices in $\R^{N\times K}$ of rank at least $r$ and having at most $s$ non-zero rows ($s$-row sparse). Let  $A\in \R^{M\times N}$  be such that the induced map $Z\mapsto AZ$ is injective on $\Sigma_{s,r}$, $X\in \R^{N\times K}$ be $s$-row sparse, $ Y=AX$ and $\rank (Y)=r$. Then 
 \begin{enumerate}
     \item $Y$ can be written as a product
	$Y=VU$ where  $V\in \R^{M\times r}$ with $\rank (V)=r$ and $U\in \R^{r\times K}$ satisfying $ UU^T=I_r$ where $I_r$ is the identity matrix in $\R^{r\times r}$.
	\item The problem 
\begin{equation}
\label{def:l0_prob}
\argmin_{Z\in \R^{N\times r}}\|Z\|_{0}\ \ \text{  s.t.  }\ \  AZ=V
\end{equation}
has a unique solution $W\in \R^{N\times r}$. This solution is $s$-row sparse and has full column rank.
\item The original matrix $X$ can be computed by $X= WU.$
 \end{enumerate}
 Here, $\|Z\|_0$ is the number of non-zero rows of matrix $Z$. 
 \end{theorem}
This result helps to reduce the computational complexity of the original problem and allows a simple implementation of joint sparse recovery algorithms for full-rank setting such as the MUSIC algorithm \cite{Schmidt79,Feng98,FY1996} to the rank deficient case. {We note that a similar reduction approach has been presented in \cite[Section III]{mishali2008reduce} { and \cite{malioutov2005sparse}}. However, in that work, the authors split the problem into two parts: first finding the support by solving the reduced problem, and then finding the unknown $X$ using the knowledge of the support, whereas we have more complete result utilizing the rank. } 

As a relaxation of the minimization problem \eqref{def:l0_prob} in the theorem above, we propose a new method for the recovery of jointly sparse vectors in the form of  a non-convex  optimization problem on the non-compact Stiefel manifold: 
\begin{equation}
\label{def:constr_prob}
\argmin_{Z\in \R^{N\times r},\; \rank(Z)=r}\|Z(Z^TZ)^{-\frac{1}{2}}\|_{2,1}\ \ \text{  s.t.  }\ \  AZ=V.
\end{equation}
{For the definition of $\ell_{2,1}$ norm, see \eqref{def:21norm}}.
Our method generalizes the $\ell_1/\ell_2$ minimization method studied in \cite{ELX13,hoyer04,KTF11,rahimi2018scale,yin2014ratio}. We demonstrate in our numerical examples that it outperforms the $\ell_{2,1}$ minimization commonly used to solve the joint sparse recovery problem {in the sense that much fewer measurements are required for accurate sparse reconstructions.} 

For numerical examples, we consider the unconstrained version of \eqref{def:constr_prob}
\begin{equation}
\label{def:unconstr_prob}
\argmin_{Z\in \R^{N\times r},\; \rank(Z)=r}\|Z(Z^TZ)^{-\frac{1}{2}}\|_{2,1}+ \frac{\lambda}{2}\|AZ-V\|^2_F,
\end{equation}
where $\|\cdot\|_{F}$ is the Frobenius norm. { 
Note that \eqref{def:unconstr_prob} is the Lagrangian form of the equivalent problem
\[\argmin_{Z\in \R^{N\times r},\; \rank(Z)=r}\|Z(Z^TZ)^{-\frac{1}{2}}\|_{2,1}\text{ s.t. } \|AZ-V\|^2_F\leq \epsilon\]
for some other parameter $\epsilon>0$, that measures the noise level in the measurements. Just like the $\epsilon$, the choices of tuning parameter $\lambda>0$  are problem dependent, with small values of $\epsilon$ corresponding to large values of $\lambda$, and often determined empirically (see \cite{dalalyan2017prediction,lederer2015don} for more detailed discussion). By choosing large $\lambda$, ensures equation \eqref{def:unconstr_prob} gives the same solution as \eqref{def:constr_prob}. } 
Currently, there are several packages available for manifold optimization (see \cite{huang2018roptlib} for a detailed list). We give preference to the Pymanopt package for its automatic differentiation capabilities \cite{JMLR2016}. It is the Python version of the original Manopt package \cite{manopt,vandereycken2013lowrank} written for MATLAB. We also smooth out the $\ell_{2,1}$ norm with the Huber function to be able to apply the conjugate gradient algorithm in Pymanopt. For users of the Manopt package, calculation of derivatives is needed, and we will explicitly compute the Euclidean gradient of the objective function in \eqref{def:unconstr_prob} as well as the gradient of its Huber regularized version.

The paper is organized as follows. For the rest of Section \ref{sec:intro}, we review the single sparse and joint sparse recovery problems. Section \ref{sec:prelim} provides background results on the rank-awareness of the $\ell_0$ minimization problem for matrices. Section \ref{sec:reduce} discusses a simple strategy to reduce the rank deficient problems to the full-rank setting. We relax the $\ell_0$ minimization to a manifold optimization problem in Section \ref{sec:manopt}. Section \ref{sec:huber} discusses the Huber regularization. Calculation of the derivative of our objective functions will be presented in Section \ref{sec:grad}. Finally, numerical example is given in Section \ref{sec:exp}.  
 
 \subsection{Sparse vector recovery problem}
Many signals of interest have (almost) sparse representation in some basis. For instance, images typically have sparse representations in wavelet bases, digital audio signals have sparse representation in Gabor systems, etc. Let  $x\in\R^N$ be the coefficients of the signal in such a basis and $y\in \R^m$ be a vector of linear measurements of the signal. Via the Riesz theorem, the relationship between $x$ and $y$ can be written as $y=Ax$ where the matrix $A$ can be explicitly computed.

In classical sparse sampling and recovery problem, the goal is to find the vector 
$x\in \R^N$ from the system
$Ax=y$ where $A\in\R^{M\times N}$, under the assumption that the vector $x$ is {sparse}. Even though the number $s$ of non-zero entries in $x$ can be much smaller than the dimension $N$, the indices of these entries are unknown, and directly solving the system by first finding the non-zero entry indices  of $x$ is an NP-hard problem \cite{Natarajan95}. Other methods have been suggested for solving it, one of the most popular is the $\ell_1$ minimization problem
$$\argmin_{z\in \R^{N}}\|z\|_{1} \ \ \text{ s.t. }\ \  Az=y,
$$
or its unconstrained version 
$$\argmin_{z\in \R^{N}} \|z\|_{1} +\frac{\lambda}{2}\|Az-y\|^2_{2}.
$$
 
It has been proved that under some conditions on the matrix $A$, specifically the null space property (NSP) and the restricted isometry property (RIP), these methods are  able to recover the vector $x$ exactly and stably.

 \subsection{The joint sparse recovery problem}\label{sec:jointsparseprob} In the joint sparse recovery problem we want to solve a matrix equation of the form 
 \begin{equation}\label{jointsparseprob}
     AX=Y, \;\; A\in \R^{M\times N},\;\; X\in\R^{N\times K}\text{ and } Y\in \R^{M\times K},
 \end{equation}
given that $X\in\R^{N\times K}$ is a $s$-row sparse matrix, i.e., if $X=(X[1],\dots,X[N])^T$ then $X[i]$'s are all zero for all indices $i$'s outside some subset of $\{1,\dots,N\}$ of size $s$.
  
   Our goal is to find a method to exactly recover $X$ from \eqref{jointsparseprob}, where the matrix $A\in \R^{M\times N}$ is chosen with small $M$. One expects that when increasing the number $K$ of unknown jointly sparse vectors, a high rank of $X$ will result in a reduction of $M$ which corresponds to the number of linear measurements done on the vectors. Define the $\ell_{2,1}$ norm of $X$
   \begin{equation}\label{def:21norm}
   \|X\|_{2,1}=\sum_{n=1}^{N}\|X[n]\|_2.
   \end{equation}
Similar to the case of single vector recovery, the following convex minimization problem has been widely used in place of directly solving for sparsest solution of \eqref{jointsparseprob} 
\begin{equation}\label{l21min}
\argmin_{Z\in \R^{N\times K}}\|Z\|_{2,1} \text{ s.t. } AZ=Y.
\end{equation}
In the worst case (uniform recovery) scenario, however, this method does not allow measurement reduction when the number of vectors is increased. As the next theorem shows, null space property is a necessary and sufficient condition for both the single sparse recovery via $\ell_1$ minimization and the joint sparse recovery via $\ell_{2,1}$ minimization. Specifically, if $A$ fails to recover an $s$-sparse vector with $\ell_1$ minimization, then it will fail to recover some $s$-row sparse matrix via \eqref{l21min}, no matter how high the rank of the unknown matrix is. As such, the $\ell_{2,1}$ minimization is known to be a \textit{rank blind} algorithm.

\begin{theorem}[\cite{LaiLiu, MR2918014}]\label{nspcond}
    Given a matrix $A\in \R^{M\times N}$, every  s-row sparse matrix $X\in \R^{N\times K}$  can be exactly recovered from \eqref{l21min} if and only if $A$ satisfies the NSP: 
    $$
    \|x_I\|_{1}<\|x_{I^c}\|_{1},\ \forall x\in \ker(A)\setminus \{\mathbf{0}\},\ \forall I\subset\{1,\dots,N\}, \text{ with } |I|\leq s,
    $$ 
    where $ x_{I} $ is the vector that we get after nullifying all the entries of $ x $ except the ones with indexes in $ I $. Furthermore, if there exists an $ x\in \ker(A)$ such that $$\|x_I\|_1>\|x_{I^c}\|_1,\; |I|\leq s,$$
 then, for any $1\leq r\leq \min\{s,K\}$, there is an s-row sparse matrix $X\in \R^{N\times K}$ with $\rank(X)=r$ which cannot be recovered from \eqref{l21min}. 
\end{theorem}

There have been several efforts to exploit the rank of the unknown matrix and design \textit{rank aware} algorithms to improve the  joint sparse recovery results. Examples of such algorithms include those based on MUSIC \cite{FY1996,lee2012subspace} and orthogonal matching pursuit \cite{MR2918014}. On the other hand, we are not aware of available rank aware algorithms based on (convex) functional optimization, CoSaMP or thresholding approaches.

One may think of resolving the rank blindness issue of $\ell_{2,1}$ minimization by replacing the $\ell_{2,1}$ norm with a different functional, i.e., consider the problem  
\begin{equation}\label{convexcond}
	\argmin_{Z\in \R^{N\times K}} \phi(Z)\;\;\text{   s.t.  } AZ=Y
\end{equation}
where $ \phi $ is a convex functional in $ \R^{N\times K} $. { However, as we point out below, no such convex functional like $\|\cdot\|_{2,1}$ will be rank aware, and we may need to change the geometry or look into non-convex optimization for such desirable property.

First, if \eqref{convexcond} has a unique solution for every $Y=AX$ with $X\in\Sigma_{s,r}$ then, for any $r\leq r^\prime \leq s$ and $X\in \Sigma_{s,r^\prime}$, it has a unique solution too. Hence \eqref{convexcond} is rank aware if it allows unique recovery in the class $\Sigma_{s,r}$ ($r\leq s$) and fails to recover  a matrix $X_0$ in $\Sigma_{s,r^\prime}$ for some $r^\prime<r$.
%
%
Observe that, every $ X\in \Sigma_{s,r} $ is the unique solution of \eqref{convexcond} if and only if
\begin{equation}\label{newinmidnsp}
\phi(X+W)>\phi(X),\ \ \forall  X\in \Sigma_{s,r},\ \forall W\in (\ker(A))^K\setminus\{\mathbf{0}\}. 
\end{equation}
And \eqref{convexcond} strongly fails to recover an $X_0\in\Sigma_{s,r^\prime}$ if and only if there exists a $W_0\in (\ker(A))^K\setminus\{\mathbf{0}\}$ such that $\phi(X_0+W_0)< \phi(X_0)$. However,  by the continuity of $\phi$ and the fact that $\Sigma_{s,r}$ is dense in $\Sigma_{s,r^\prime}$ when $r\geq r^\prime$, \eqref{newinmidnsp} implies that $\phi(X_0+W_0)\geq \phi(X_0)$, a contradiction. As a result, \eqref{convexcond} cannot recover any matrix in $\Sigma_{s,r}$ while strongly fail to recover in the class $\Sigma_{s,r^\prime}$ for $r^\prime \leq r$.}
  
In this paper, we utilize the rank of $Y$ to extract information about  $X$ and reduce the problem to a smaller dimensional problem in full-rank setting. {As a first step towards a rank aware optimization approach for uniform recovery, we develop a new non-convex optimization method on manifold, which achieves a significant reduction in the number of measurements in joint sparse recovery, compared to $\ell_{2,1}$ minimization.} 


\section{Preliminary results}
\label{sec:prelim}
This section revisits the $\ell_0$ minimization for matrices, based on results in \cite{ChenHuo06,MR2918014}. We present a uniqueness condition for the sparse recovery on $\Sigma_{s,r}$, that is, the map $Z\mapsto AZ$ is injective on $\Sigma_{s,r}$. Before we characterize this condition, let us make a few observations. 


\begin{definition}
The spark of the matrix $A$, denoted by $ \spark( A )$, is the smallest number $j$ such that there exists $j$ linearly dependent columns in $A$.
\end{definition}

Note that any $ \spark( A ) -1$ columns of $A$ are linearly independent.

\begin{proposition}
\label{prop:prelim}
If $A:\Sigma_{s,r}\to \R^{M\times K}$ is injective then for every $X\in \Sigma_{s}$, $\rank(X)=\rank(AX).$
\end{proposition}
\begin{proof}
We prove this proposition into two steps: 
\begin{enumerate}[\ \ \ \ \  i)]
\item if $A:\Sigma_{s,r}\to \R^{M\times K}$ is injective then $ \spark(A)> s $,
\item if  $\,\spark(A)> s $ then $\rank (X)=\rank (AX)$, $\forall  X \in \Sigma_s $. 
\end{enumerate}

To prove i), assume otherwise ($ \spark(A)\le s $), there exists an $ I\subseteq \{1,\dots,N\} $, $|I|=s$ and vectors $ v_1,\dots,v_r\in \R^N $ which are linearly independent and supported on $ I $ such that $ Av_1=0 $. Take the matrices  $$ X=(v_1,\cdots,v_r,\overbrace{0,\cdots,0}^{K-r}) ,\;\; \widetilde X=(-v_1,\cdots,v_r,\overbrace{0,\cdots,0}^{K-r}) .$$ Then $ X,\widetilde X\in \Sigma_{s,r} $, $ X\neq \widetilde X $ and $ AX=A\widetilde X $, which contradicts the injectivity of $A$. 

For ii), suppose $\rank(X)>\rank(AX)$ for $X\in \Sigma_s$. Then, there exists a vector $x\neq 0$ in the span of columns of $X$ such that $Ax=0$. $x$ is $s$-sparse, hence we found a non-trivial combination of $s$ columns of $A$ that is equal to $0$, contradicting the assumption that $\spark(A)>s$.
\end{proof}
 
The uniqueness condition for the sparse recovery on $\Sigma_{s,r}$ is characterized in the following theorem. 
 
\begin{theorem}\label{rankrecequiv}
The following statements are equivalent:
\begin{enumerate}
 \item $A:\Sigma_{s,r}\to \R^{M\times K}$ is injective,
\item $ s\leq\dfrac{\spark(A)-1+r}{2} $,
\item For every $ X\in \Sigma_{s,r}$, $ X $ is the unique solution of the minimization problem 
	\begin{equation}\label{equivrankcond}
		\min_{Z\in \R^{N\times K}} \|Z\|_{0}\text{  s.t.  } AZ=Y , 
	\end{equation}
	where $ Y=AX $.
\end{enumerate}
\end{theorem}	

\begin{proof}
$ (1)\Leftrightarrow (2) $ follows as an intermediate consequence of \cite[Theorem 2]{MR2918014}.

	$ (1)\Rightarrow(3) $ Suppose  $A:\Sigma_{s,r}\to \R^{M\times K}$ is injective. From Proposition \ref{prop:prelim}, $\rank(Y)=\rank(X)\ge r$. Let $ \widetilde X $ be a solution of \eqref{equivrankcond}, then $ \|\widetilde X\|_0\leq \|X\|_0=s $. Also, since $ A\widetilde X=Y $, $\rank(\widetilde X)\geq \rank(Y) \ge r$, hence $\widetilde X\in \Sigma_{s,r}$. From the injectivity of $A$, we have $X=\widetilde X$, and $X$ is the unique solution of \eqref{equivrankcond}.
	
	
	$ (3)\Rightarrow(1) $  Suppose  $X,\widetilde X\in \Sigma_{s,r}$ and $ AX=A\widetilde X=Y $. Then both of them are the unique solutions of the same minimization problem, therefore $ X=\widetilde X $. 
\end{proof}

Theorem \ref{rankrecequiv} shows that the $\ell_0$ minimization \eqref{equivrankcond} is rank aware. Indeed, matrices $X$ with larger rank allow $A$ with smaller $\spark(A)$, therefore can be reconstructed from fewer measurements. Alternatively, for a fixed $A$, the upper bound of the sparsity of unknown matrices that can be exactly recovered grows linearly with its rank. In general, the $\ell_0$ norm as a functional is not suitable for computations. We present a relaxation of this problem using manifold optimization approach in Section \ref{sec:manopt}. Immediately below is a useful idea to reduce matrix equations to full-rank setting.  

\section{Reduction to full-rank setting}
\label{sec:reduce}
We present a simple way to transform the matrix equation to be solved into one where the unknown is full-rank matrix. At the expense of an additional decomposition of output data, the benefit of this practice is twofold. First, we reduce the dimension of the matrix equation and now find solutions in a potentially much lower dimensional space $\R^{N\times r}$, compared to the original one in $\R^{N\times K}$. More importantly, this step allows us to apply and exploit the strength of several joint sparse recovery methods which are particularly powerful in the full-rank case to the more general setting with possible rank defect. 

Our main result in this section is the following. 
\begin{theorem}
\label{reductionthm}
Let {$X\in \R^{N\times K}$ be $s$-row sparse}, $A\in \R^{M\times N}$ such that $A:\Sigma_{s,r}\to \R^{M\times K}$ is injective and $ Y=AX$ {with $\rank (Y)=r$}.
 Then 
 \begin{enumerate}[\ \ \ \ \ i)]
     \item $Y$ can be written as a product
	$Y=VU$ where  $V\in \R^{M\times r}$ with $\rank (V)=r$ and $U\in \R^{r\times K}$ satisfying $ UU^T=I_r$ where $I_r$ is the identity matrix in $\R^{r\times r}$.
	\item The problem 
\begin{equation}
 \label{fullrankequiveq1}
\argmin_{Z\in \R^{N\times r}}\|Z\|_{0}\text{  s.t.  } AZ=V
\end{equation}
has a unique solution $W\in \R^{N\times r}$. This solution is $s$-row sparse and has full column rank.
\item The original matrix $X$ can be computed by $X= WU.$
 \end{enumerate}
 \end{theorem}

\begin{proof} Since $\rank(X) = r$, we have from Proposition \ref{prop:prelim} that $\rank(Y)= r$. The decomposition of $Y$ in i) is well-known. We note that this decomposition is not unique but one such representation can be computed from the compact singular value decomposition of $Y$ or the RQ decomposition.

For ii), using the fact that $ UU^T=I_r$, we get $V=YU^T=AXU^T$. Also, $XU^T$ is $s$-row sparse and $\rank(XU^T)=\rank(X)=r$. From Theorem \ref{rankrecequiv}, $W=XU^T$ is the unique solution to  \eqref{fullrankequiveq1}.
 
For iii), $AWU=VU=Y$ and $\|WU \|_0=\|X\|_0$, hence $X= WU$.
\end{proof}
{\begin{remark}\label{rmrk:factor}
Other decompositions $Y=VU$ of $Y$ can also be considered. As long as the columns of $V$ span the same subspace, we will have the same result as above. 
\end{remark}}

{Theorem \ref{reductionthm} shows that instead of directly solving problem \eqref{equivrankcond}, one can solve the reduced problem \eqref{fullrankequiveq1}, where the unknown matrix has full rank. Then, the solution of original problem can be obtained with one matrix multiplication. Note that for actual data that contains noise, we can compute the SVD decomposition of $Y$ and threshold the eigenvalues to get the essential rank of the matrix. Moreover, if $N$ is too large, and we have an estimate on the sparsity $s$, we may consider  the truncated SVD decomposition instead with first $s$ largest singular values. This may reduce the computational cost dramatically when $s\ll N$.

Notice that if $\widetilde W$ is an approximate solution of \eqref{fullrankequiveq1}, since the columns of $U$ make up a Parseval frame, we have $\|WU-\widetilde WU\|_{F}=\|W-\widetilde W\|_{F}$. Thus the approximation error in solving the reduced problem does not result in a larger (in $\ell_2$ sense) error for the solution of the original problem.  }
  
%
%

\section{Non-convex manifold optimization for joint sparse recovery}
\label{sec:manopt} 
In this section, we present a novel relaxation to $\ell_0$ minimization in form of an optimization problem on manifold. 

Consider problem \eqref{fullrankequiveq1}. As known from the previous section that its solution must have full rank, we can restrict this problem to the space of all full-rank matrices in $\R^{N\times r}$. This adds a rank constraint to the joint sparsity recovery, and our idea is to recast the optimization
problem with full-rank condition in the Euclidean space into one on a manifold that encodes this constraint (i.e., a non-compact Stiefel manifold), \cite{absil2009optimization}. Furthermore, we relax the $\ell_0$ norm of the solution with the concept of orthogonal factor of matrices.

 \begin{definition}
             Let $Z$ be a matrix in $\R^{N\times r}$ with $\rank(Z)=r$, we call $Z(Z^T\! Z)\!^{-\frac{1}{2}}$ its orthogonal factor. 

\end{definition}
       
The orthogonal factor is the product of orthogonal factors in the compact SVD decomposition of $Z$.
Note that $\|Z\|_0=\|Z(Z^TZ)^{-\frac{1}{2}}\|_{0}$, it is quite natural to relax $\|Z\|_0$ with $\|Z(Z^TZ)^{-\frac{1}{2}}\|_{2,1}$. Define the non-compact Stiefel manifold 
$$
\St(N,r)  = \{Z\in \R^{N\times r}: \rank(Z)=r\},
$$ 
we seek to approximate the solution to \eqref{fullrankequiveq1} via the following manifold optimization problem
\begin{equation}
\label{newopt}
\argmin_{Z\in\, \St(N,r)}\|Z(Z^TZ)^{-\frac{1}{2}}\|_{2,1}\ \ \text{  s.t.  } \ \ AZ=V. 
\end{equation}
\begin{remark}$ $
\label{funcinterpret}
\begin{enumerate}
    \item  $\|Z(Z^TZ)^{-\frac{1}{2}}\|_{2,1}$ can also be interpreted as follows: if we take any orthonormal basis for the range of $Z$, then the $\ell_{2,1}$ norm of that basis is invariant of the choice of the basis and exactly equal to $\|Z(Z^TZ)^{-\frac{1}{2}}\|_{2,1}$. 
    \item $\|Z(Z^TZ)^{-\frac{1}{2}}\|_{2,1}$ is a bounded functional and therefore not geodesically convex. A bounded functional is convex on a Riemannian manifold if and only if it is constant (\cite[Corollary 2.5]{udriste}). 
    \item { 
Since proposed manifold optimization method is non-convex, the outcomes of gradient based optimization methods for solving it largely depend on the quality of the initialization. As such, the convex $\ell_{2,1}$ minimization can be used to obtain a good initial value for the manifold method, which will be then used to arrive at a better reconstruction using the conjectured rank-awareness property of the manifold method. } 
\end{enumerate}
\end{remark}

%
 \subsection{Particular cases}
 \label{sec:particular_cases}
Some particular cases of \eqref{newopt} are of interest on their own. For $r=s$, there holds $\|W(W^TW)^{-\frac{1}{2}}\|_{2,1}=s$ with $W$ being the true solution to \eqref{newopt}. 
Indeed, from Remark \ref{funcinterpret},  $\|W(W^TW)^{-\frac{1}{2}}\|_{2,1}$ is the $\ell_{2,1}$ norm of $s$ orthonormal basis vectors that have joint sparsity $s$. Since each of the non-zero rows of this matrix forms a unit vector in $\ell_2$ norm, its $\ell_{2,1}$ norm is equal to $s$.

 For $r=1$, \eqref{newopt} is equivalent to minimizing the $\ell_1/\ell_2$ functional:  
$$		\argmin_{z\in \R^N} \|z\|_{1}/\|z\|_2\ \ \text{  s.t.  } \ \ Az=v. 
$$

This functional has been shown to enhance sparsity of the solutions and outperform $\ell_1$ on many test problems \cite{hoyer04,KTF11,ELX13,yin2014ratio}. However, rigorous theory on whether $\ell_1/\ell_2$ is superior to $\ell_1$ in uniform recovery, specifically, whether the NSP is sufficient for the exact recovery using $\ell_1/\ell_2$ is largely open. One particular challenge is that $\ell_1/\ell_2$ penalty is neither convex nor concave nor separable, therefore could not be treated under current analysis, see \cite{TranWebster17}. Extending to the general case ($r > 1$), we conjecture that the NSP for the manifold optimization problem \eqref{newopt} is less demanding than that for $\ell_{2,1}$ minimization (which is identical to standard NSP, see Theorem \ref{nspcond}). The theoretical verification of this postulation is a subject of ongoing research. 

{ We remark that other functionals are possible. In this paper, we mainly demonstrate the effective performance of the manifold optimization method, leaving open the question of optimality of the proposed functionals.} 

\section{Huber regularization}
\label{sec:huber}
We consider the following unconstrained problem, {which produces the same solution as \eqref{newopt} for large $\lambda$:}
	   \begin{equation}\label{unconstrmanopt}
\argmin_{Z\in\, \St(N,r)} \|Z(Z^TZ)^{-\frac{1}{2}}\|_{2,1} + \frac{\lambda}{2} \|AZ-V\|^2_F.
\end{equation}
{Satisfactory choices of $\lambda$ are problem dependent and often determined empirically, as have been done herein.}
With notice that $\ell_{2,1}$ norm is non-smooth, we further relax $\|\cdot\|_{2,1}$ with Huber regularization to be able to apply manifold optimization methods, which have been mostly developed for smooth optimization. In particular, we replace the $\|X[n]\|_2$  with $H_\delta(\|X[n]\|_2)$ where $H_\delta\in C^1$ is the Huber function
	    $$H_\delta(x)=\begin{cases}
	    x-\delta/2, & x\geq \delta,\\
	    \frac{x^2}{2\delta}, & 0\le  x< \delta,
	    \end{cases}$$
(see Figure \ref{fig:huber}) and consider the problem
\begin{equation}\label{unconstrmanopthuber}
\argmin_{Z\in\, \St(N,r)}  F_\delta(Z) + \frac{\lambda}{2}\|AZ-V\|^2_F,
\end{equation}
where 
$$F_\delta(Z)=\sum_{n=1}^N\left[H_\delta(\|Z(Z^TZ)^{-\frac{1}{2}}[n]\|_2)+\frac{\delta}{2}\right].$$
%
The following proposition reveals that \eqref{unconstrmanopthuber} is an effective surrogate for problem \eqref{unconstrmanopt}. 
\begin{proposition}
If  \eqref{unconstrmanopt} has a unique solution and there exists $\delta_0 > 0$ such that \eqref{unconstrmanopthuber} has unique solution for all $\delta < \delta_0$, then the solution to \eqref{unconstrmanopthuber} converges to the solution to \eqref{unconstrmanopt} as $\delta$ converges to $0$.
\end{proposition}
\begin{proof}
$\|Z(Z^TZ)^{-\frac{1}{2}}\|_{2,1}$ is a continuous functional on the compact Stiefel manifold and
 $F_\delta(Z)$
 is decreasing with respect to $\delta$ and pointwise converges to  $\|Z(Z^TZ)^{-\frac{1}{2}}\|_{2,1}$. Thus, from \cite[Proposition 5.7]{dal2012introduction}, $F_\delta(Z)$ $\Gamma$-converges to  $\|Z(Z^TZ)^{-\frac{1}{2}}\|_{2,1}$. 
\end{proof}

{ Unlike the $\lambda$ parameter, there is no trade-off for the parameter $\delta$ in the Huber function per se. The smaller it is, the closer the recovered solution will be to the true solution. However the trade-off will be introduced by the fact that we are using a smooth method (conjugate gradient) to solve a non-smooth problem. If the parameter is too small we may end up with  a non-smooth problem which we are trying to avoid. }

We remark that it is possible to obtain a differentiable objective from \eqref{unconstrmanopt} with other strategies, such as using dummy variables, see \cite{ELX13} for a development of this approach for $\ell_1/\ell_2$ minimization. The conjugate gradient (CG) algorithm is applied to solve the Huber regularization problem \eqref{unconstrmanopthuber}. Since this is a non-convex problem, the iteration may be trapped at some local minimum. For example, when $X$ is of full column rank ($r=s$), we know from Section \ref{sec:particular_cases} that the global minimum of the objective function is equal to $s$; however, the CG algorithm fails to find it in some cases. To deal with this difficulty, we run the algorithm with different starting points  that are randomly selected in $\mathbb{S}(N,r)$ using random normally distributed matrices. In our numerical tests, the global minima can be reached this way after only a few choices of initial values. Certainly, for large-scale problems, advanced methods like simulated annealing \cite{KGV83,Harjek88} can be employed to reduce the number of starting values required. 

\begin{figure}
\begin{center}
		\includegraphics[scale=0.09]{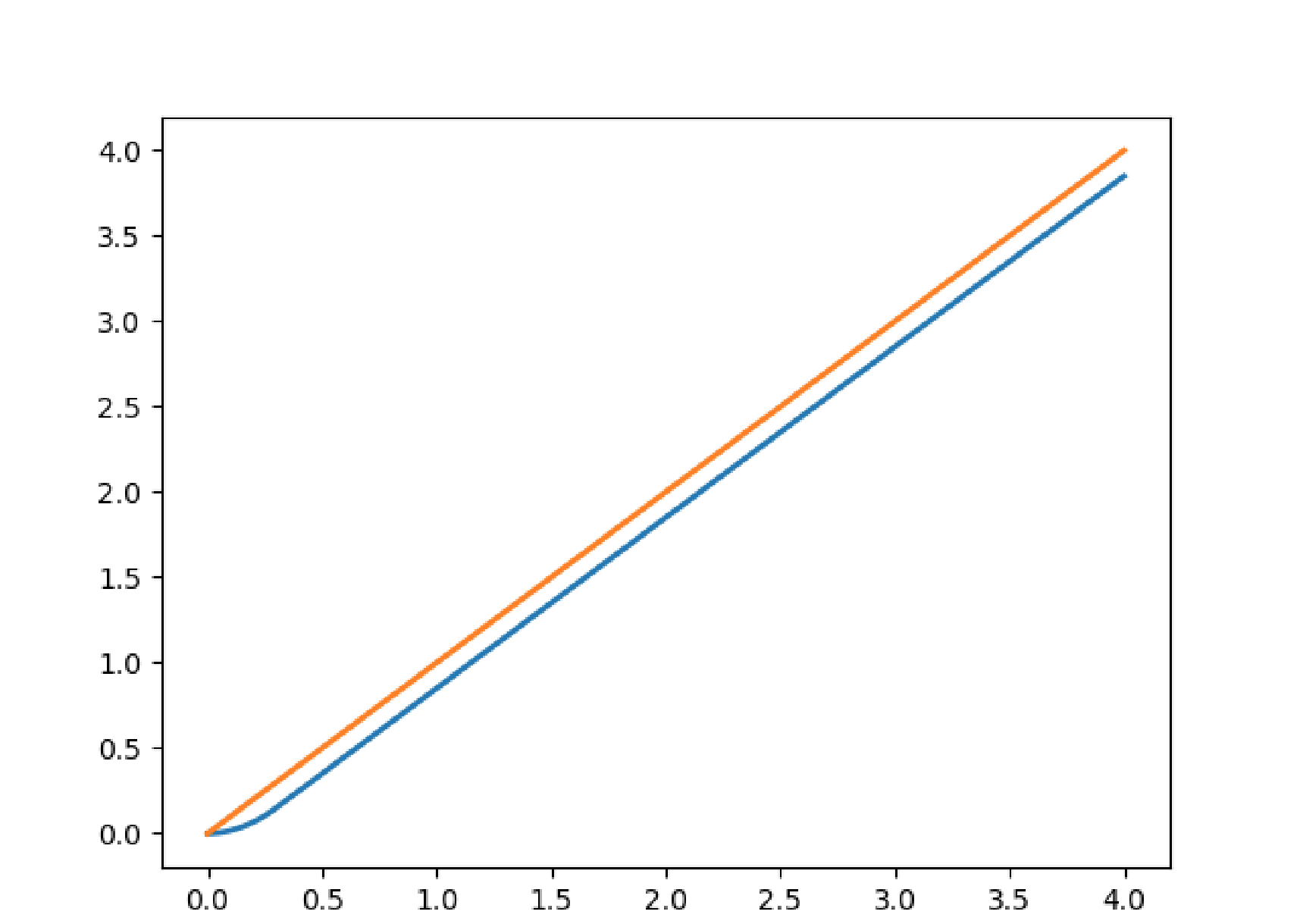}

		\caption{Huber function for $\delta=0.3$ and the $y=x$ function. }
		\label{fig:huber}
\end{center}
\end{figure}

\begin{figure}
\begin{center}
			\includegraphics[scale=0.35]{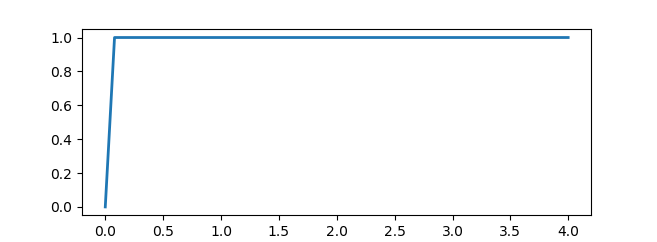}

		\caption{Derivative of Huber function. }
		\label{fig:huber_der}
\end{center}
\end{figure}
%

\section{Explicit formulae for the Euclidean gradients of cost functions}
\label{sec:grad}
The Euclidean gradients are the fundamental building block for the optimization of cost functions on manifolds. Such derivatives are often taxing to calculate, posing a major hurdle for implementing manifold optimization methods for many problems, especially where large matrices or vectors are involved. Interestingly, for our objective functions \eqref{unconstrmanopt} and \eqref{unconstrmanopthuber}, it is possible to derive the exact, explicit formulae for their gradients. Expectedly, these formulae save the users a considerable amount of time on computing either the numerical or symbolic differentiation.


Throughout this section, $\nabla \equiv \nabla_Z$, the derivative with respect to $Z$. Recall $I_N$ is the identity matrix in $\R^{N\times N}$. It can be easily checked that 
$$\nabla\|AZ-V\|_F^2=2A^T(AZ-V),$$
therefore we will focus on computing the gradients  $\nabla \|Z(Z^TZ)^{-\frac{1}{2}}\|_{2,1}$ and $\nabla F_\delta(Z)$.

Let us introduce some additional notation. For a matrix $Z\in\St(N,r)$, denote by $D(Z)$  and $D_\delta(Z)$ the  diagonal matrices in   $\R^{N\times N}$  with 
$$D(Z)_{i,i}=\! \begin{cases}
\frac{1}{\|Z[i,:]\|_2}, &\|Z[i,:]\|_2 \neq 0,\\
0, & \text{otherwise,}
\end{cases}
$$
$$
\ \ \text{ and }\ \  D_\delta(Z)_{i,i}=\! \begin{cases}
\frac{H_\delta^\prime(\|Z[i,:]\|_2)}{\|Z[i,:]\|_2}, &\|Z[i,:]\|_2 \neq 0,\\
0, & \text{otherwise,}
\end{cases} $$
 where $H^\prime_\delta(x)$ is the derivative of the Huber function
$$H^\prime_\delta(x)=\begin{cases}
1, & x\geq \delta,\\
\frac{x}{\delta}, & 0\leq x< \delta.
\end{cases}$$ 
Also denote by $\odot $ the Schur-Hadamard (a.k.a., elementwise) product between two matrices of the same size. The following proposition provides an explicit formula for the cost function \eqref{unconstrmanopt}. 
\begin{proposition}\label{prop:gradoffunc}
Let $Z=U\Lambda V$ be the compact singular value decomposition of $Z \in  \St(N,r)$. Then 
$$\nabla \|Z(Z^TZ)^{-\frac{1}{2}}\|_{2,1}=(I_N-UU^T)D(UV)U\Lambda^{-1}V.$$
\end{proposition}
\begin{proof}
For a scalar function $f: \R\to \R$, by an abuse of notation, the matrix function induced by $f$ is defined on $\St(N,r)$ as $f(Z)=U\Lambda_f V$, where $\Lambda_f$ is the $r \times r$ diagonal matrix such that
$$
(\Lambda_f)_{ii} = f(\Lambda_{ii}). 
$$
Note that $f(Z) = Z(Z^TZ)^{-\frac{1}{2}}$ for $f\equiv 1$. 
Let $g(Z)=\|Z\|_{2,1}$, our goal is to compute $\nabla (g\circ f)$.

Let $L_g(Z,E)$ denote the Fr\'echet  derivative of $g$ at $Z$ in the direction $E\in \R^{N\times r}$ (and similarly for $f$, $g \circ f$, etc.). We will compute $\nabla (g\circ f)$ via
\begin{equation}
\label{eq:est1}
\tr(\nabla (g\circ f)\cdot E^T) = L_{g\circ f}(Z,E). 
\end{equation}
First, by the chain rule 
$$
L_{g\circ f}(Z,E)=L_{g}(f(Z),L_f(Z,E)).$$
Notice that $\nabla g=D(Z)\cdot Z$, hence
$$
L_{g}(f(Z),L_f(Z,E))  =\tr(D(f(Z))\cdot f(Z)\cdot L_f(Z,E)^T)
    =\tr(D(UV)\cdot f(Z)\cdot L_f(Z,E)^T).
$$

Let $F(Z) \in \R^{r\times r}$ and $G(Z)\in \R^{(N-r)\times r}$ be defined by 
$$ F(Z)_{i,j}=\begin{cases}
\frac{1}{\lambda_i+\lambda_j}, &i\neq j,\\
0, & i=j,
\end{cases}
\ \ \ \
\text{ and }
\ \ \ \
G(Z)_{i,j}=\frac{1}{\lambda_j},
$$
where $\lambda_j$ is the $j$-th singular value of $Z$.
Let $Z=\left[\begin{matrix}
U\;
W
\end{matrix} \right]\left[\begin{matrix}
\Lambda\\
0
\end{matrix} \right] V$ be the full SVD of $Z$.
From \cite[Corollary 3.12]{noferini2017formula}, 
\begin{equation*}
\begin{split}
L_f(Z,E) & =\left[\begin{matrix}
U\;
W
\end{matrix} \right]
\left( \left[\begin{matrix}
F\\
G
\end{matrix} \right]\odot
\left[\begin{matrix}
U^TEV^T\\
W^TEV^T
\end{matrix} \right]\
 +
 \left[\begin{matrix}
-F\\
0
\end{matrix} \right]\odot
 \left[\begin{matrix}
VE^TU\\
W^TEV^T
\end{matrix} \right]\
 \right) V
\\
& = U( F\odot (U^TEV^T-VE^TU)) V+W(G\odot (W^TEV^T))V 
\\
& =U( F\odot (U^TEV^T-VE^TU)) V+WW^TEV^T\Lambda^{-1}V
\\
& =U( F\odot (U^TEV^T-VE^TU)) V+(I_N-UU^T)EV^T\Lambda^{-1}V,
\end{split}
\end{equation*}
using the fact that $G\odot A=A\Lambda^{-1}$, $\forall A \in \R^{(N-r)\times r}$ and $WW^T=I_N-UU^T$.
Thus we have that
\begin{equation}
\label{eq:est2}
\begin{split}
L_{g\circ f}(Z,E) & = \tr(D(UV) U  V \cdot V^T [ F\odot (U^TEV^T-VE^TU)]^T U^T)
\\
&\qquad\qquad\qquad + \tr(D(UV)UVV^T\Lambda^{-1}VE^T(I_N-UU^T))
\\
& = \tr(U^T D(UV)U [F\odot (U^TEV^T-VE^TU)]^T)
\\
&\qquad\qquad\qquad +\tr(D(UV)U\Lambda^{-1}VE^T(I_N-UU^T))
\\
& = \tr((U^TEV^T-VE^TU) \cdot [F\odot ( U^T D(UV)U )]^T)
\\
&\qquad\qquad\qquad +\tr((I_N-UU^T)D(UV)U\Lambda^{-1}VE^T).
\end{split}
\end{equation}
Here we used the fact that $\tr(A\cdot (B\odot C)^T)=\tr(C\cdot (B\odot A)^T).$
Continuing our computations 
\begin{equation*}
\begin{split}
& \tr((U^TEV^T-VE^TU) \cdot [F\odot ( U^T D(UV)U )]^T)
\\
= \ & \tr(U^TEV^T \cdot [F\odot (U^T D(UV)U )]^T)-\tr(VE^TU\cdot [F\odot (U^T D(UV)U )]^T)
\\
= \ & \tr([F\odot (U^T D(UV)U) ] \cdot VE^TU )-\tr(U [F\odot (U^T D(UV)U) ]^TV E^T )
\\
= \ & \tr(U([F\odot (U^T D(UV)U) ] -[F\odot (U^T D(UV)U) ]^T)VE^T )=0, 
\end{split}
\end{equation*}
which, in combining with \eqref{eq:est1}--\eqref{eq:est2}, proves our assertion.
\end{proof}

Explicit formula for the cost function \eqref{unconstrmanopthuber} can be easily derived from the above proof.

\begin{proposition}
Let $Z=U\Lambda V$ be the compact singular value decomposition of $Z\in \St(N,r)$. Then 
$$\nabla F_\delta(Z)=(I_N-UU^T)D_\delta(UV)U\Lambda^{-1}V.$$
\end{proposition}
\begin{proof}
Notice that the functional 
$$g_\delta(Z)=\sum_{n=1}^N\left[H_\delta(\|Z[n]\|_2)+\frac{\delta}{2}\right]$$
has the gradient $\nabla g_\delta=D_\delta(Z)\cdot Z$ and the rest of the computations follows the same arguments as in Proposition \ref{prop:gradoffunc}.
\end{proof}

	   \section{Numerical experiments with Gaussian data and measurements}
	   \label{sec:exp}
	   	   
We use the Pymanopt package \cite{JMLR2016}, which is the Python version of the original Manopt package \cite{manopt,vandereycken2013lowrank} for MATLAB, in these numerical tests. 
The { conjugate gradient method} with backtracking Armijo line-search provided in Pymanopt is used to solve \eqref{unconstrmanopthuber}. For automatic differentiation, we employ the Autograd package. In each iteration, an $r\times r$ SVD  is computed. For sparse recovery, since $r\leq s$, this step is not numerically expensive.
The initial value is chosen randomly. 
	   
In this test we compare the number of measurements required by $\ell_{2,1}$ minimization and our method in sparse matrix recovery. We generate a random Gaussian matrix $A$ of size $80\times 300$, and form measurement matrices $A_k$ by taking the first $k$ rows of $A$. We also generate a random Gaussian matrix $X$ of sparsity $s=30$ of size $300\times 70$ for the solution and test the recovery of $X$ given the output $Y_k = A_k X$, for each $k\in \{38,40,42,\dots,80\}$. (Hence, $k$ stands for the number of measurements). The test is rerun 22 times for different realizations of $A$ and $X$. We perform the $\ell_{2,1}$ minimization using the spgl1 package \cite{BergFriedlander:2008} {(which is specialized to solve $\ell_1$ and related regularized least square problems),} and our method using Pymanopt, { with $\delta=0.001$ and $\lambda=9$,} until relative reduction of the gradient norm is $10^{-8}$ or 1000 iterations have been run. For each value of $k$, we plot the scattered errors and median of recovery errors in log scale over $22$ trials. The results of the experiment are presented in Figure \ref{fig:num_meas}. As can be seen, our manifold optimization algorithm starts converging with much fewer measurements compared to the Euclidean minimization with the $\ell_{2,1}$ norm.

\begin{figure}
\begin{center}
\includegraphics[scale=0.34]{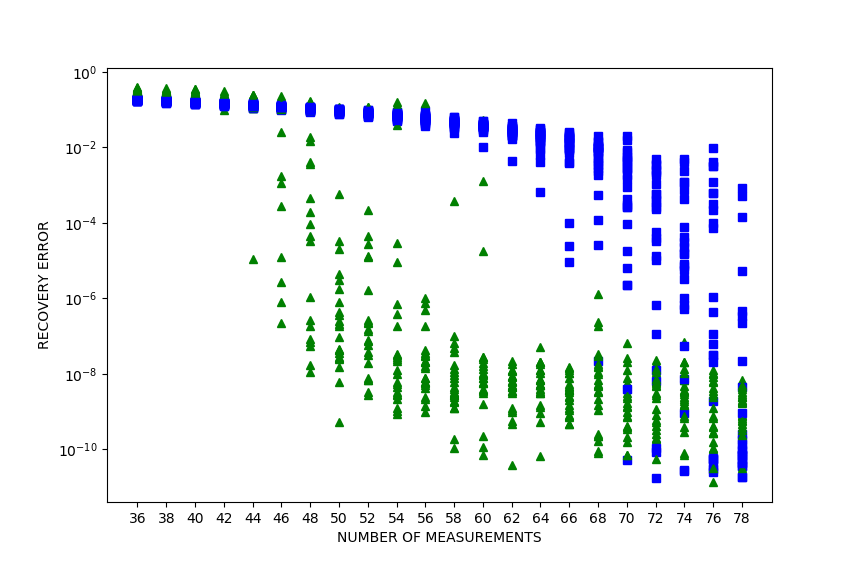}
\includegraphics[scale=0.365]{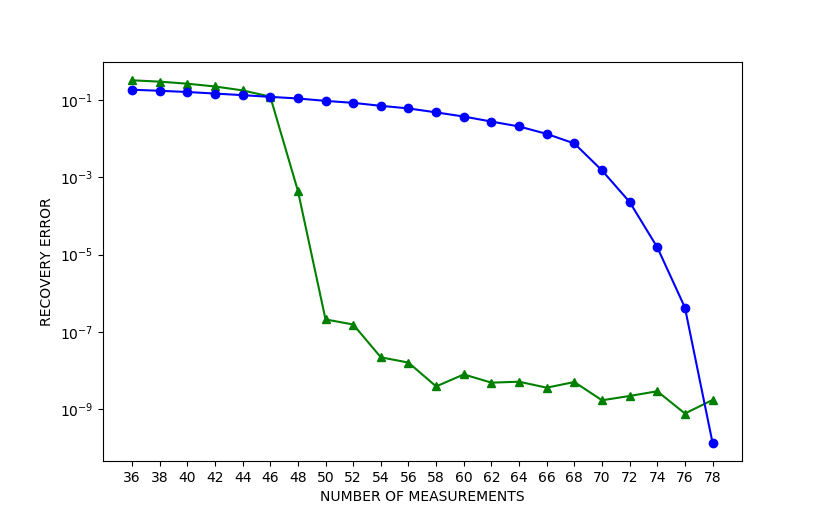}

\caption{Scattered errors and median of errors in log scale vs. number of measurements for $\ell_{2,1}$ (circular labels) and manifold optimization \eqref{unconstrmanopthuber} (triangular labels). }	
\label{fig:num_meas}
\end{center}
\end{figure}
 
\section{Conclusion and future work}

{We study the joint sparse recovery problem. Our result here allows to reduce the  problem to a potentially much smaller dimensional problem where the unknown matrix has full column rank. We also offer a new functional minimization method on the manifold of full-rank matrices for the recovery of the unknown matrix.  As pointed out in Section \ref{sec:jointsparseprob}, we cannot expect to have a rank sensitive convex optimization method on the Euclidean space, so changing the geometry and looking into manifold optimization is probably a good direction to go. {The functional we suggest does encourage sparsity and in our numerical experiment outperforms the Euclidean $\ell_{2,1}$ minimization.}
We pose the following questions as potential directions for future investigation.
\begin{enumerate}[\ \ \ \ \  (a)]
\item 
If  $A\in \R^{M\times N}$ satisfies the $\text{NSP}$ condition, then is it true that any $s$-sparse $x\in \R^N\setminus \{0\}$  can be uniquely recovered by solving 
$$		\argmin_{z\in \R^N} \|z\|_{1}/\|z\|_2\text{  s.t.  } Az=Ax?
$$
 More generally, does $\text{NSP}$ imply unique recovery of every $s$-row sparse $W\in \R^{N\times r}$ from \eqref{newopt}, for any $1\leq r\leq s$?
\item 
Is \eqref{newopt}
 rank aware? More specifically, are there measurement matrices $A$ that recover every $W\in \Sigma_{s,r}$ by solving \eqref{newopt} for large $r$'s but not for small $r$'s?
\item Does there exist a sparsity-promoting geodesically convex manifold optimization method for joint sparse recovery?
\item Investigation of robustness and stability of our method in presence of noise. 
\end{enumerate}

 }
\label{sec:conclusion}

\section{Acknowledgments} We want to thank Richard Barnard (Western Washington University) for suggesting the use of Huber regularization and 
Carola-Bibiane Sch\"onlieb (University of Cambridge) for pointing out the $\Gamma$-convergence argument for Huber regularization to us. { We also thank the anonymous referee for valuable suggestions.}

This work is supported by 
 the National Science Foundation, Division of Mathematical Sciences, Computational Mathematics program under contract number 
  DMS1620280;
  the U.S. Department of Energy, Office of Science, Office of Advanced Scientific Computing Research, Applied Mathematics program under contract number ERKJ314; Scientific Discovery through Advanced Computing (SciDAC) program through the FASTMath Institute under Contract No. DE-AC02-05CH11231; and, the Laboratory Directed Research and Development program at the Oak Ridge National Laboratory, which is operated by UT-Battelle, LLC., for the U.S. Department of Energy under Contract DE-AC05-00OR22725.

\bibliographystyle{model1b-num-names}
\bibliography{references}

\end{document}